\documentclass[12pt]{amsart}%

\usepackage{amsmath,amssymb,amsthm,amscd}
\usepackage{hyperref}
\usepackage{color}

\theoremstyle{plain}
\newtheorem{theorem}{Theorem}[section]
\newtheorem{proposition}{Proposition}[section]

\newtheorem{lemma}{Lemma}[section]

\theoremstyle{remark}
\newtheorem{definition}{Definition}[section]
\newtheorem{remark}{Remark}[section]
\newtheorem{examples}{Examples}[section]
\newtheorem{assumption}{Assumption}[section]

\DeclareMathOperator{\supp}{supp}

\begin{document}

\title[Dirichlet forms under the Gelfand transform]{Measures and Dirichlet forms under the Gelfand transform}
\author[Hinz]{Michael Hinz$^{1,2}$}
\address{Mathematisches Institut, Friedrich-Schiller-Universit\"at Jena, Ernst-Abbe-Platz 2, 07737, Germany and Department of Mathematics,
University of Connecticut,
Storrs, CT 06269-3009 USA}
\email{Michael.Hinz.1@uni-jena.de and Michael.Hinz@uconn.edu}
\thanks{$^1$Research supported in part by the Alexander von Humboldt Foundation Feodor (Lynen Research Fellowship Program)}
\author[Kelleher]{Daniel Kelleher$^2$}
\address{Department of Mathematics, University of Connecticut, Storrs, CT 06269-3009 USA}
\email{Daniel.Kelleher@uconn.edu}
\thanks{$^2$Research supported in part by NSF grant DMS-0505622}
\author[Teplyaev]{Alexander Teplyaev$^2$}
\address{Department of Mathematics, University of Connecticut, Storrs, CT 06269-3009 USA}
\email{Alexander.Teplyaev@uconn.edu}

\dedicatory{Dedicated with deep respect to Professor
Ildar Ibragimov on the occasion of his eightieth birthday.}

\date{\today}

\begin{abstract}
Using the standard tools of Daniell-Stone integrals, Stone-\v{C}ech compactification and Gelfand transform, we discuss how any Dirichlet form defined on a measurable space can be  
transformed into a regular Dirichlet form on a locally compact space. This implies  existence, on the Stone-\v{C}ech compactification, of the associated  Hunt process. 
 As an application, we show that for any separable resistance form in the sense of Kigami there exists an associated Markov process. 
\tableofcontents
\end{abstract}
\maketitle

\section{Introduction}\label{S:Intro} 

The main object of our study is a 
 Dirichlet forms $(\mathcal{E},\mathcal{F})$ on the $L_2$-space over a measure space $(X,\mathcal{X},\mu)$. The notion of the Dirichlet form means that $\mathcal{E}$   is a closed nonnegative (bilinear) quadratic form on  $L^2(X,\mathcal{X},\mu)$ with a dense domain $\mathcal{F}\subset L^2(X,\mathcal{X},\mu)$.  Moreover $(\mathcal{E},\mathcal{F})$ has  what is called Markov (or positivity preserving, or normal contraction property): if $u\in \mathcal{F}$ then $\bar u = min(u,1)\in \mathcal{F}$ and 
 $$\mathcal E(\bar u,\bar u) \leqslant \mathcal E(u,u).$$
 By the combination of the standard theories of quadratic forms on Hilbert spaces,  the spectral theory of self-adjoint operators and the Hille-Yosida theorem, there exists an associated 
 self-adjoint operator (non-negative or non-positive, depending on the analytic or probabilistic conventions), which generates a positivity preserving contraction semigroup on $L^2(X,\mathcal{X},\mu)$. 
 This is equivalent to having a semigroup of transition probability kernels which, 
 by the Kolmogorov's general theory of random process, is equivalent to the existence of a symmetric Markov process (in the usual way one may have to allow for the extinction of the process, or to augment the state space $X$ with a  ``cemetery'' point).  
This set up has generated an abundance of strong and well-known results, see e.g. \cite{BH91,Dy1,Dy,FOT94,MR92,RS,RW}, and recently was extensively used in analysis and probability on fractals, see \cite{Ki,Str,K}. 
However most of the  basic  results in the theory of Dirichlet forms and Markov processes rely on a set up where $X$ is assumed to be a topological space. Examples include the classical Beurling-Deny decomposition for regular Dirichlet forms, the existence of energy measures in the sense of Fukushima \cite{FOT94} and LeJan \cite{LJ78} or the existence of an associated Hunt process. To discuss them most references require $X$ to be locally compact. Of course it is desirable to have versions of these theorems in more general situations (for instance for quasi-regular Dirichlet forms on Souslin spaces), and therefore a reduction of topological assumptions was one of the various directions into which the standard theory for regular Dirichlet forms has been extended. One of the typical strategies is to embed the possibly non-locally compact state space $X$ into a larger but (locally) compact space and to transfer the Dirichlet form to this new space, where the standard theory for the locally compact case applies.
See for instance \cite{AMR, AR89, AR90, ChF12, F89, MR92} for some applications of such compactification methods. In \cite{AMR} this idea was used to prove a Beurling-Deny type theorem for quasi-regular Dirichlet forms on Hausdorff spaces $X$ that are such that each compact is metrizable and its Borel $\sigma$-algebra is countably generated. 
Historically the representation theoretic point of view upon Dirichlet forms already dates back to the work of Beurling and Deny, \cite{BD, BDb}, and was later taken up by Fukushima in connection with regularization techniques, see \cite{Fu71} and \cite[Appendix A.4]{FOT94}. A study of Banach algebras naturally induced by Dirichlet forms was carried out by Cipriani, \cite{C06}.

The main ideas of the present note are not new but versions of these ideas. In contrast to the mentioned references we do not assume the given state space $X$ to carry any topology (except for Section \ref{S:seppoints}), and one item we would like to highlight in this context is the Daniell-Stone representation theorem, see e.g. \cite{Dudley}.
Given a multiplicative Stonean vector lattice $\mathcal{B}$ of bounded real-valued functions on a set $X$ we use the connection between the Daniell-Stone theorem and Gelfand's representation theorem for $C^\ast$-algebras to establish an injection of a suitable class of measures on $X$ into the space of nonnegative Radon measures on the spectrum $\Delta$ of the complex uniform closure of $\mathcal{B}$. We apply this idea to show that for any given Dirichlet form over a measurable space there is a corresponding uniquely determined regular Dirichlet form on a larger and locally compact state space.

 We consider the algebra $\mathcal{B}(\mathcal{E})$ of bounded measurable functions on $(X,\mathcal{X})$ that are $\mu$- square integrable and have 
 finite energy. The uniform closure of its complexification is a $C^\ast$-algebra, and its spectrum $\Delta$ is a locally compact Hausdorff space. If $\mathcal{B}(\mathcal{E})$ vanishes nowhere, then $\Delta$ (roughly speaking) contains $X$ as a dense subset, and there is a Radon measure $\hat{\mu}$ on $\Delta$ which is uniquely determined by $\mu$ in a way that makes the restriction of the Gelfand transform $f\mapsto \hat{f}$ to $\mathcal{B}(\mathcal{E})$ an $L_2$-isometry, i.e.
\begin{equation}\label{E:introL2}
\left\|\hat{f}\right\|_{L_2(\Delta,\hat{\mu})}=\left\|f\right\|_{L_2(X,\mu)},\ \ f\in\mathcal{B}(\mathcal{E}).
\end{equation}
This allows to define a symmetric bilinear form by
\[\hat{\mathcal{E}}(\hat{f},\hat{g}):=\mathcal{E}(f,g), \ \ f,g\in\mathcal{B}(\mathcal{E}).\]

Our main result, Theorem \ref{T:closable}, says that $\hat{\mathcal{E}}$, together with the image $\hat{\mathcal{B}}(\mathcal{E})$ of $\mathcal{B}(\mathcal{E})$ under the Gelfand map, is closable, and its closure $(\hat{\mathcal{E}},\hat{\mathcal{F}})$ in $L_2(\Delta,\hat{\mu})$ is a symmetric regular Dirichlet form. In other words, we can find a locally compact Hausdorff space $\Delta$ which 'contains' the state space $X$, and a regular Dirichlet form $(\hat{\mathcal{E}},\hat{\mathcal{F}})$ that is the image of $(\mathcal{E},\mathcal{F})$. For this Dirichlet form we can now apply the standard theory \cite{FOT94} and for instance obtain a Beurling-Deny representation and the existence of energy measures. We would like to point out that in \cite{AR89} the embedding of a Souslin standard Borel space into the Gelfand spectrum of a countably generated and point separating algebra of continuous functions had been used to construct a symmetric Hunt process associated with the given Dirichlet form.

In contrast to references like \cite{AMR, AR89} it may not be possible to pull these results back to the Dirichlet form $(\mathcal{E},\mathcal{F})$ on the original state space $X$. For instance, the energy measure of $(\hat{\mathcal{E}},\hat{\mathcal{F}})$ on $\Delta$ may be such that the image of $X$ under the embedding into $\Delta$ is of zero energy measure, see Example~\ref{Ex:counter}. This is reminiscent of the situation in infinite dimensional analysis where the Cameron-Martin space typically is a null set, cf. Remark \ref{R:infdim} and such references as \cite{Gross1,Gross2,IR,Sudakov}. The study of the Dirichlet form $(\hat{\mathcal{E}},\hat{\mathcal{F}})$ on the spectrum $\Delta$ may be a natural way to enlarge the space to support energy measures. Under additional topological assumptions we can recover results similar to those in \cite{AMR, AR89}.

Before we turn to Dirichlet forms we discuss how to naturally relate suitable measures $\mu$ on $X$ to Radon measures $\hat{\mu}$ on $\Delta$. This correspondence relies on a connection between the Daniell-Stone theorem and the Gelfand transform. Although this idea is not new, see for instance \cite{Fuchsst77}, it does not seem to be all too widely used. We consider a multiplicative vector lattice $\mathcal{B}$ of bounded real-valued functions on $X$. The uniform closure $A(\mathcal{B})$ of its complexification is a commutative $C^\ast$-algebra. If $\mu$ is uniquely associated with a positive linear functional on $\mathcal{B}$ then we may use positivity arguments to obtain a uniquely associated positive linear functional on the space $C_c(\Delta, \mathbb{R})$ of real-valued compactly supported functions on the spectrum of $A(\mathcal{B})$. By the Riesz representation theorem this functional can be represented by integration with respect to some uniquely determined Radon measure $\hat{\mu}$ on $\Delta$. Proceeding this way we obtain an injective mapping from a cone of nonnegative measures on $X$ into the cone of nonnegative Radon measures on $\Delta$. The isomorphism property of the Gelfand transform finally yields the $L_2$-isometry (\ref{E:introL2}).

The paper is organized as follows. For convenience, we recall some preliminaries concerning Gelfand theory and the Daniell-Stone theorem in the section \ref{S:Daniell}. In Section \ref{S:connect} we investigate the connection for multiplicative Stonean vector lattices of bounded real-valued functions and establish some lemmas on positivity, support properties and denseness. The main result of Section \ref{S:functionals} is Theorem \ref{T:transfermeasure}, which states the correspondence between measures on $X$ and $\Delta$. As a consequence we also obtain the $L_2$-isometry (\ref{E:introL2}). In Section \ref{S:Dirichlet} we apply these results to Dirichlet forms to obtain the closability of $(\hat{\mathcal{E}},\hat{\mathcal{B}}(\mathcal{E}))$ in $L_2(\Delta,\hat{\mu})$ and the regularity of its closure $(\hat{\mathcal{E}},\hat{\mathcal{F}})$, Theorem \ref{T:closable}. Consequences include the Beurling-Deny representation and the existence of Radon energy measures for the transferred Dirichlet form $(\hat{\mathcal{E}},\hat{\mathcal{F}})$ on $\Delta$, sketched in Section \ref{S:energymeasures}.

We write $C_0(\Delta)$ to denote the space of continuous functions on $\Delta$ that vanish at infinity and $C_c(\Delta)$ to denote its subspace of functions with compact support. For their subspaces of real-valued functions we write $C_0(\Delta,\mathbb{R})$ and $C_c(\Delta,\mathbb{R})$, respectively, and we will do similarly for other function spaces. If the index set of a sequence is not specified, it is the set of natural numbers, and if corresponding limits are taken, they are taken with the index going to infinity.

\subsection{Acknowledgements}
Helpful discussions with Mikhail Gordin, Masha Gordina, Naotaka Kajino, Jun Kigami and Takashi Kumagai are gratefully acknowledged.

\section{Gelfand theory and the Daniell-Stone Theorem}\label{S:Daniell}

For multiplicative vector lattices of bounded real valued functions the theorem of Daniell-Stone can be connected to Gelfand's representation theorem for commutative $C^\ast$-algebras. In this section we briefly recall these two concepts.

We start with remarks on \emph{commutative Gelfand theory}, cf. \cite{Arveson, Blackadar, Kaniuth}. Let $A$ be a commutative $C^\ast$-algebra of bounded functions $a:X\to \mathbb{C}$, with the supremum norm $\left\|\cdot\right\|$ and with the algebra operations defined pointwise and the involution $^\ast$ defined by complex conjugation $a^\ast:=\overline{a}$. By $\Delta(A)$ we denote the \emph{spectrum (Gelfand space)} of $A$, the space of continuous, complex-valued, multiplicative functionals on $A$. Equipped with the Gelfand topology the spectrum $\Delta(A)$ becomes a regular locally compact Hausdorff space, cf. \cite{Kaniuth}. If $A$ contains the constant function $\mathbf{1}$ then $\Delta(A)$ is compact. The space $\Delta(A)$ is second countable if and only if the $C^\ast$-algebra $A$ is separable, and this in turn is equivalent to $A$ being countably generated. For any $a\in A$ the \emph{Gelfand transform} $\hat{a}:\Delta(A)\to \mathbb{C}$ of $a$ is defined by $\hat{a}(\varphi):=\varphi(a)$, and by the Gelfand representation theorem the Gelfand map $a\mapsto \hat{a}$ is seen to be an isometric $^\ast$-isomorphism from the Banach algebra $A$ onto the algebra $C_0(\Delta(A))$ of continuous functions on $\Delta(A)$ vanishing at infinity. If the algebra $A$ vanishes nowhere on $X$, that is, if for any $x\in X$ there exists some $a\in A$ such that $a(x)\neq 0$, then $X$ may be identified with a subset of $\Delta(A)$ by the map $\iota: X\to\Delta(X)$, where 
\begin{equation}\label{E:iota}
\iota(x)(a):=a(x)\ ,\ \ a\in A,
\end{equation}
for any $x\in X$. Note that multiplication in $C_0(\Delta(A))$ is given pointwise, and 
\[\iota(x)(a_1a_2)=(a_1a_2)(x)=a_1(x)a_2(x)=\iota(x)(a_1)\iota(x)(a_2)\]
for any $x\in X$ and $a_1,a_2\in A$. Thus, we observe the set-theoretic inclusion $\iota(X)\subset \Delta(A)$.
The set $\iota(X)$ is dense in $\Delta(A)$. For if not, we could find a nonzero function $f\in C_0(\Delta(A))$ such that $f(\iota(x))=0$ for all $x\in X$. Then, however, some nonzero $a\in A$ would have to exist with $\hat{a}=f \in C_0(\Delta(A))$, hence 
\begin{equation}\label{E:zero}
\hat{a}(\iota(x))=\iota(x)(a)=a(x)
\end{equation} 
would have to be zero for all $x\in X$ and consequently $a\equiv 0$ in $A$, a contradiction.

The second tool we would like to sketch is the \emph{Daniell-Stone Theorem}. Let $X\neq 0$ and let $\mathcal{L}$ be a real vector lattice of functions on $X$, i.e. a vector space of functions $f:X\to\mathbb{R}$ that is closed under minimum and maximum operations $f\wedge g=\min(f,g)$ and $f\vee g=\max(f,g)$. We assume that $\mathcal{L}$ possesses the \emph{Stone property}: for any $f\in\mathcal{L}$, $f\wedge 1\in\mathcal{L}$. By $\sigma(\mathcal{L})$ we denote the $\sigma$-ring of subsets of $X$ generated by $\mathcal{L}$ and by $\mathcal{M}^+(\sigma(\mathcal{L}))$, the cone of (nonnegative) measures on $\sigma(\mathcal{L})$. A positive linear functional $I:\mathcal{L}\to\mathbb{R}$ is called a \emph{Daniell integral} on $\mathcal{L}$ if for any sequence $(f_n)_n\subset\mathcal{L}$ of nonnegative functions decreasing to zero pointwise at all $x\in X$ also the sequence of integrals $(I(f_n))_n$ decreases to zero. The Daniell-Stone Theorem says that for any Daniell integral $I$ on $\mathcal{L}$ there exists a uniquely determined measure $\mu \in \mathcal{M}^+(\sigma(\mathcal{L}))$ on  $\sigma(\mathcal{L})$ such that 
\begin{equation}\label{E:Daniellrep}
I(f)=\int_Xfd\mu\ , \ \ f\in\mathcal{L}. 
\end{equation}
See for instance \cite{Dudley}. We use the notation
\[\mathcal{D}(\mathcal{L}):=\left\lbrace \mu\in \mathcal{M}^+(\sigma(\mathcal{L})): \text{all functions from $\mathcal{L}$ are $\mu$-integrable }\right\rbrace.\]
If $I$ is a Daniell integral on $\mathcal{L}$ then the measure $\mu$ uniquely associated with $I$ by (\ref{E:Daniellrep}) is a member of $\mathcal{D}(\mathcal{L})$. Conversely any $\mu\in \mathcal{D}(\mathcal{L})$ defines a Daniell integral on $\mathcal{L}$ by (\ref{E:Daniellrep}).
Note that if $\mathcal{L}$ contains a strictly positive function, then all measures in $\mathcal{D}(\mathcal{L})$ are $\sigma$-finite, and if it  contains the constant function $\mathbf{1}$, then all measures in $\mathcal{D}(\mathcal{L})$ are finite.\\

\section{Multiplicative Stonean vector lattices}\label{S:connect}

We are interested in special cases to which both theories apply. Let $\mathcal{B}$ be a real \emph{multiplicative} vector lattice of \emph{bounded} functions on $X\neq \emptyset$ that has the Stone property. By $\mathcal{B}+i\mathcal{B}$ we denote its complexification, that is the complex vector space of functions $f_1+if_2$ with $f_1,f_2\in\mathcal{B}$. The vector space operations and the complex conjugation are defined pointwise. We endow $\mathcal{B}+i\mathcal{B}$ with the supremum norm $\left\|\cdot\right\|$ and denote its closure by $A(\mathcal{B})$, clearly a Banach space. Pointwise multiplication turns $A(\mathcal{B})$ into a commutative Banach algebra, and with the involution $^\ast$ defined by complex conjugation it becomes a commutative $C^\ast$-algebra. Under the Gelfand transform $f\mapsto \hat{f}$ the $C^\ast$-algebra  $A(\mathcal{B})$ is isometrically $^\ast$-isomorphic to $C_0(\Delta(A(\mathcal{B})))$. To shorten notation we will write $\Delta$ to abbreviate $\Delta(A(\mathcal{B}))$. From now on we will assume the following.

\begin{assumption}\label{A:nonvanish}
The space $\mathcal{B}$ vanishes nowhere.
\end{assumption}
Under this assumption the set $\iota(X)$, where $\iota$ is defined as in (\ref{E:iota}) with $A=A(\mathcal{B})$, is a dense subset of $\Delta$, and according to (\ref{E:zero}) we have $\hat{f}(\iota(x))=f(x)$ for any $f\in\mathcal{B}$ and $x\in X$. 

To discuss nonnegativity issues let $A(\mathcal{B})^+$ and $C_0(\Delta)^+$ denote the cones of real-valued nonnegative functions in $A(\mathcal{B})$ and $C_0(\Delta)$, respectively. For a real-valued function $f$ we write $f^+=\max(f,0)$ and $f^-=\max(-f,0)$. If $f$ is a member of $\mathcal{B}$ then so are $f^+$ and $f^-$. 

\begin{lemma}\label{L:positive}
A function $f\in A(\mathcal{B})$ is real-valued if and only if $\hat{f}\in C_0(\Delta)$ is. Moreover, we have $f \in A(\mathcal{B})^+$ if and only if $\hat{f}\in C_0(\Delta)^+$.
\end{lemma}

This lemma is a consequence of (\ref{E:zero}) together with the denseness of $\iota(X)$ in $\Delta$.

\begin{lemma}\label{L:standard}
For any real-valued $f\in A(\mathcal{B})$ we have $(f^+)^\wedge=\hat{f}^+$ and $(f^-)^\wedge=\hat{f}^-$. 
\end{lemma}

\begin{proof}
For any $x\in X$ we have $(f^+)^\wedge(\iota(x))=f^+(x)$ by (\ref{E:zero}). If $f(x)\geq 0$ then $f^+(x)=f(x)=\hat{f}(\iota(x))=\hat{f}^+(\iota(x))$. If $f(x)<0$ then   $\hat{f}(\iota(x))<0$ and  $\hat{f}^+(x)=0$. Consequently $(f^+)^\wedge(\iota(x))=\hat{f}^+(\iota(x))$ for all $x\in X$, and by linearity also $(f^-)^\wedge(\iota(x))=\hat{f}^-(\iota(x))$. By continuity and the denseness of $\iota(X)$ in $\Delta$ the lemma follows.
\end{proof}

The members of $A(\mathcal{B})^+$ are all monotone limits of nonnegative functions from $\mathcal{B}$. For this statement Assumption \ref{A:nonvanish} is not needed.

\begin{lemma}\label{L:approxf}
For any function $f\in A(\mathcal{B})^+$ there exists a monotonically increasing sequence $(f_n)_n$ of nonnegative functions $f_n\in\mathcal{B}$ that converges to $f$ pointwise. 
\end{lemma}
\begin{proof}
By the lattice property in $\mathcal{B}$, we can see that there is a sequence $(g_n)_n$ of nonnegative functions $g_n\in\mathcal{B}$ converging uniformly to $f$. 
We may assume that the nonnegative numbers $\delta_n:=\sup_X|g_n-g_{n+1}|$ are such that $\sum_n\delta_n<\infty$ (otherwise pass to a subsequence). Setting $f_n:=g_n-g_n\wedge(\sum_{k=n}^\infty\delta_k)$ we obtain a sequence $(f_n)_n$ with the desired properties.
\end{proof}

We discuss compactly supported functions. If $\mathcal{B}$ contains the constant functions, then $\Delta$ is compact, hence every function in $\hat{\mathcal{B}}$ has compact support. To formulate a result for the general case, set
\[\mathcal{B}_c:=\left\lbrace \varphi\in\mathcal{B}:\hat{\varphi}\in C_c(\Delta)\right\rbrace.\]
Clearly $\mathcal{B}_c$ is again a multiplicative vector lattice having the Stone property.

\begin{lemma}\label{L:dense}
The space $\mathcal{B}_c$ is uniformly dense in $\mathcal{B}$.
\end{lemma}

To prove Lemma \ref{L:dense} we use a property of upper level sets. Given $\varphi\in\mathcal{B}$ and $k\in\mathbb{N}\setminus\left\lbrace 0\right\rbrace$ set 
\[N_k(f):=\left\lbrace x\in X:|f(x)|\geq \frac{1}{k}\right\rbrace.\]
\begin{lemma}\label{L:abovelevel}
For any $f\in\mathcal{B}$ and any $k$ the closure of the set $\iota(N_k(f))$ is compact in $\Delta$.
\end{lemma}
\begin{proof}
We have $|f|\in\mathcal{B}$ and, according to Lemma \ref{L:standard}, $|f|^\wedge=|\hat{f}|$. Consequently we may assume $f\geq 0$. Since $\hat{f}\in C_0(\Delta)$, the closed set 
\[L_k(f):=\left\lbrace y\in \Delta: \hat{f}(y)\geq \frac{1}{k}\right\rbrace\]
is contained in a compact set and therefore compact itself. On the other hand $\iota(N_k(f))\subset L_k(f)$, what implies that 
$\overline{\iota(N_k(f))}$ is a closed subset of $L_k(f)$, hence compact.
\end{proof}

We verify Lemma \ref{L:dense}. 

\begin{proof}
It suffices to show that nonnegative functions can be approximated. Given $f\in\mathcal{B}$ with $f\geq 0$ consider the functions
\[\varphi_k:=f-f\wedge \frac{1}{k}.\]
Obviously the sequence $(\varphi_k)_k$ uniformly converges to $f$, and for fixed $k$ the set
\[N^k:=\left\lbrace x\in X: \varphi_k(x)>0\right\rbrace\]
is a subset of $N_k(f)$. On the other hand, we have
\[\left\lbrace y\in \Delta: \hat{\varphi}_k(y)>0\right\rbrace\subset \overline{\iota(N^k)}.\]
For if there were some $y\in \Delta$ with $\hat{\varphi}_k(y)>0$ having an open neighborhood $U_y$ such that $\varphi_k(x)=0$ for all $x\in X$ with $\iota(x)\in U_y$, then we would have $\hat{\varphi}_k(z)=0$ for all $z\in U_y$ by the density of $\iota(X)$ in $\Delta$, a contradiction. It also follows that
\[\supp\hat{\varphi}_k\subset \overline{\iota(N^k)}\subset \overline{N_k(f)},\] 
and Lemma \ref{L:abovelevel} implies that $\supp\hat{\varphi}_k$ is compact.
\end{proof}

\section{Positive linear functionals and measures}\label{S:functionals}

In this section we establish a correspondence between suitable measures $\mu$ on $X$ and Radon measures $\hat{\mu}$ on $\Delta$ and list some consequences. As before we assume that $\mathcal{B}$ is a Stonean multiplicative vector lattice of bounded real-valued functions on $X$.

Let $I:\mathcal{B}\to\mathbb{R}$ be a positive linear functional. Given a function $f\in A(\mathcal{B})^+$ and an increasing sequence $(f_n)_n\subset\mathcal{B}$ of nonnegative function as in Lemma \ref{L:approxf}, we set
\begin{equation}\label{E:extendI}
I(f):=\sup_n I(f_n).
\end{equation}
The lattice property of $\mathcal{B}$ guarantees that (\ref{E:extendI}) provides a well-defined positive linear (i.e. positively homogeneous and additive) functional $I:A(\mathcal{B})^+\to [0,+\infty]$. In what follows let Assumption \ref{A:nonvanish} be satisfied. In view of Lemma \ref{L:positive} we can then define a bounded positive linear functional $\hat{I}:C_0(\Delta)^+\to [0,+\infty]$ by
\begin{equation}\label{E:deffunctional}
\hat{I}(\hat{f}):=I(f), \ \ \hat{f}\in C_0(\Delta)^+,
\end{equation}
and according to Lemma \ref{L:standard} we may set $I(f):=I(f^+)-I(f^-)$ and $\hat{I}(\hat{f}):=I(f)$ to extend (\ref{E:deffunctional}) to all $f\in \mathcal{B}$. 

Let $\mathcal{M}^+(\Delta)$ denote the cone of nonnegative Radon measures on $\Delta$. The Riesz representation theorem ensures the existence of a uniquely determined $\hat{\mu}\in\mathcal{M}^+(\Delta)$ such that for any $\hat{f}\in C_c(\Delta)$ we have
\begin{equation}\label{E:Rieszrep}
\hat{I}(\hat{f})=\int_{\Delta}\hat{f}d\hat{\mu} .
\end{equation}

\begin{remark}2
Recall that to prove the existence part of the Riesz representation theorem one usually sets
\[\hat{\mu}(K):=\inf\left\lbrace \hat{I}(\hat{f}): f\in C_c(\Delta,\mathbb{R}), \text{ and $f\geq \mathbf{1}_K$}\right\rbrace\]
for compact $K\subset\Delta$ and defines the $\hat{\mu}$-measure of an arbitrary Borel set by inner approximation by compacts. It is therefore sufficient to know the functional $\hat{I}$ on the cone $C_c(\Delta)^+$.
\end{remark}

Now assume that $I:\mathcal{B}\to\mathbb{R}$ is a Daniell integral and $\mu \in \mathcal{D}(\mathcal{B})$ is the unique measure on $\sigma(\mathcal{B})$ associated with $I$ as in (\ref{E:Daniellrep}). In this case definition (\ref{E:deffunctional}) yields
\begin{equation}\label{E:integralscoincide}
\int_Xf d\mu=\int_{\Delta}\hat{f}d\hat{\mu}, \ \ f\in \mathcal{B}.
\end{equation}
The map $\mu\mapsto \hat{\mu}$ is positive and linear (i.e. additive and positively homogeneous). By (\ref{E:integralscoincide}) and the uniqueness part of the Daniell-Stone Theorem we obtain the following result.

\begin{theorem}\label{T:transfermeasure}
The map $\mu\mapsto \hat{\mu}$ is an injection of $\mathcal{D}(\mathcal{B})$ into $\mathcal{M}^+(\Delta)$.
\end{theorem}

We may also consider equivalence classes of functions.

\begin{lemma}
Let $\mu\in \mathcal{D}(\mathcal{B})$ and $f\in \mathcal{B}$. Then $f=0$ $\mu$-a.e. on $X$ if and only if $\hat{f}=0$ $\hat{\mu}$-a.e. on $\Delta$. 
\end{lemma}
\begin{proof}
Let $f=0$ $\mu$-a.e. on $X$. Then also $f^+$ and $f^-$ vanish $\mu$-a.e. on $X$. By Lemma \ref{L:standard} and (\ref{E:integralscoincide}) therefore $\int_\Delta \hat{f}^+ d\hat{\mu}=0$, hence $\hat{f}^+=0$ $\hat{\mu}$-a.e. The same is true for $\hat{f}^-$ and consequently $\hat{f}=0$ $\hat{\mu}$-a.e. The converse implication follows in a similar manner.
\end{proof}

Therefore the Gelfand map induces a well-defined map from the space of $\mu$-equivalence classes of functions from $\mathcal{B}$ into the space of $\hat{\mu}$-equivalence classes of functions on $\Delta$. We denote it again by $f\mapsto \hat{f}$. We investigate corresponding $L_2$-spaces.

\begin{lemma}\label{L:L2}
Let $\mu\in \mathcal{D}(\mathcal{B})$. For $f\in\mathcal{B}$ we have
\[\left\|f\right\|_{L_2(X,\mu)}=\left\|\hat{f}\right\|_{L_2(\Delta,\hat{\mu})}.\]
\end{lemma}
\begin{proof}
Being an algebra homomorphism, the Gelfand map satisfies $(\hat{f})^2=(f^2)^\wedge$ for any $f\in A(\mathcal{B})$. For $f\in \mathcal{B}$ the identity (\ref{E:integralscoincide}) then yields 
\[\int_X f^2d\mu=\int_\Delta (f^2)^\wedge d\hat{\mu}=\int_\Delta \hat{f}^2d\hat{\mu}.\]
\end{proof}

The following fact will be used in the next section.

\begin{lemma}
For any $\mu\in \mathcal{D}(\mathcal{B})$ the image $\hat{\mathcal{B}}$ of $\mathcal{B}$ is dense in $L_2(\Delta,\hat{\mu},\mathbb{R})$.
\end{lemma}

\begin{proof}
Since $C_0(\Delta, \mathbb{R})$ is a dense subspace of $L_2(\Delta,\hat{\mu},\mathbb{R})$, it suffices to show that any $\hat{f}\in C_0(\Delta, \mathbb{R})$ can be approximated in $L_2(\Delta,\hat{\mu},\mathbb{R})$ by functions from $\hat{\mathcal{B}}$. However, as $C_0(\Delta)$ is isometrically isomorphic to the uniform closure $A(\mathcal{B})$ of the complexification of $\mathcal{B}$, there is a sequence $(f_n)_n\subset \mathcal{B}$ such that $(\hat{f}_n)_n$ approximates $\hat{f}$ uniformly. Given $\varepsilon>0$ we can find a compact set $K_\varepsilon\subset\Delta$ such that $\hat{\mu}(\Delta\setminus K_\varepsilon)<\varepsilon$. Then obviously
\[\lim_n \int_{K_\varepsilon} |\hat{f}_n-\hat{f}|^2d\hat{\mu}=0\]
and 
\[\int_{\Delta\setminus K_\varepsilon}|\hat{f}_n-\hat{f}|^2d\hat{\mu}\leq \varepsilon (\left\|f\right\|+\sup_n\left\| f_n\right\|).\]
\end{proof}

\section{Dirichlet forms under the Gelfand map}\label{S:Dirichlet}

We use the setup of the previous section to transfer from a Dirichlet form on a measure space to a regular Dirichlet form on a locally compact second countable Hausdorff space. 

Let $(X,\mathcal{X},\mu)$ be a measure space and $(\mathcal{E},\mathcal{F})$ a Dirichlet form on $L_2(X,\mu,\mathbb{R})$, see for example \cite[Chapter I]{BH91}. We will frequently use the shorthand notation $\mathcal{E}(f):=\mathcal{E}(f,f)$ and do similarly for other bilinear expressions. The space of bounded measurable functions on $X$ is denoted by $b\mathcal{X}$. Set
\begin{equation}\label{E:BE}
\mathcal{B}(\mathcal{E}):=\left\lbrace f\in b\mathcal{X}: \text{ the $\mu$-equivalence class of $f$ is in $\mathcal{F}\cap L_1(X,\mu,\mathbb{R})$}\right\rbrace.
\end{equation}
The Cauchy-Schwarz inequality and the Markov property of $(\mathcal{E},\mathcal{F})$ imply that $\mathcal{B}(\mathcal{E})$ is a multiplicative vector lattice that has the Stone property. In addition we assume the following:

\begin{assumption}\label{A:nonvanishenergy}
The space $\mathcal{B}(\mathcal{E})$ vanishes nowhere on $X$ and is separable with respect to the supremum norm.
\end{assumption}

Let $\Delta$ be the spectrum of the uniform closure $A(\mathcal{B}(\mathcal{E}))$ of the complexification of $\mathcal{B}$. For $f,g\in \mathcal{B}(\mathcal{E})$ we set
\begin{equation}\label{E:Dformhat}
\hat{\mathcal{E}}(\hat{f},\hat{g}):=\mathcal{E}(f,g).
\end{equation}
Obviously $\hat{\mathcal{E}}$ is a nonnegative definite symmetric bilinear form on the dense subspace 
\[\hat{\mathcal{B}}(\mathcal{E})=\left\lbrace \hat{f}\in C_0(\Delta,\mathbb{R}): f\in \mathcal{B}(\mathcal{E})\right\rbrace \]
of $L_2(\Delta,\hat{\mu},\mathbb{R})$. It enjoys the Markov property. In fact, it defines a regular symmetric Dirichlet form on $L_2(\Delta,\hat{\mu},\mathbb{R})$.

\begin{theorem}\label{T:closable}
The form $(\hat{\mathcal{E}},\hat{\mathcal{B}}(\mathcal{E}))$ is closable on $L_2(\Delta,\hat{\mu},\mathbb{R})$. Its closure $(\hat{\mathcal{E}}, \hat{\mathcal{F}})$ defines a symmetric regular Dirichlet form.
\end{theorem}

\begin{proof}
Let $(\hat{f}_n)_n$ be a sequence of functions from $\hat{\mathcal{B}}(\mathcal{E})$ that is $\hat{\mathcal{E}}$-Cauchy and tends to zero in $L_2(\Delta,\hat{\mu},\mathbb{R})$. Then by (\ref{E:Dformhat}) the sequence $(f_n)_n$ of preimages $f_n\in \mathcal{B}(\mathcal{E})$ of the functions $\hat{f}_n$ under the Gelfand map is $\mathcal{E}$-Cauchy, and by Lemma \ref{L:L2} it tends to zero in $L_2(X,\mu)$. From the closability of $(\mathcal{E},\mathcal{F})$ together with (\ref{E:Dformhat}) it then follows that 
\[\lim_n \hat{\mathcal{E}}(\hat{f}_n)=\lim_n\mathcal{E}(f_n)=0.\]
Therefore $(\hat{\mathcal{E}},\hat{\mathcal{B}}(\mathcal{E}))$ is closable. According to Lemma \ref{L:dense} the set 
\[\hat{\mathcal{B}}_c(\mathcal{E}):=\left\lbrace \hat{f}\in C_c(\Delta): f\in \mathcal{B}(\mathcal{E})\right\rbrace\]
is uniformly dense in $\hat{\mathcal{B}}(\mathcal{E})$, hence also in $C_0(\Delta)$. On the other hand, given $f\in\mathcal{B}(\mathcal{E})$, the functions
\[\varphi_k:=f-(f\vee (-\frac{1}{k}))\wedge \frac{1}{k}\]
converge to $f$ in $\mathcal{E}_1$-norm, see for instance \cite[Theorem 1.4.2]{FOT94}. Consequently $\hat{\mathcal{B}}_c(\mathcal{E})$ is a core for $(\hat{\mathcal{E}}, \hat{\mathcal{F}})$. Note that as a consequence of Assumption \ref{A:nonvanishenergy} the Gelfand spectrum $\Delta$ of $\mathcal{A}(\mathcal{B}(\mathcal{E}))$ is second countable.
\end{proof}

To the symmetric regular Dirichlet form $(\hat{\mathcal{E}}, \hat{\mathcal{F}})$ on $L_2(\Delta,\hat{\mu},\mathbb{R})$ we refer as the \emph{transferred Dirichlet form}.

\section{Beurling-Deny decomposition and energy measures}\label{S:energymeasures}

We record some consequences of the existing theory for Dirichlet forms on locally compact spaces when applied to $(\hat{\mathcal{E}}, \hat{\mathcal{F}})$. As before let $(X,\mathcal{X},\mu)$ be a measure space and $(\mathcal{E},\mathcal{F})$ a symmetric Dirichlet form on $L_2(X,\mu)$ such that Assumption \ref{A:nonvanishenergy} is satisfied.

The first theorem is the \emph{Beurling-Deny representation}.

\begin{theorem}
The transferred Dirichlet form $(\hat{\mathcal{E}},\hat{\mathcal{F}})$ on $L_2(\Delta,\hat{\mu},\mathbb{R})$ admits the decomposition
\begin{multline}
\hat{\mathcal{E}}(\hat{f},\hat{g})=\hat{\mathcal{E}}^c(\hat{f},\hat{g})+\int\int_{\Delta\times\Delta}(\hat{f}(x)-\hat{f}(y))(\hat{g}(x)-\hat{g}(y))\hat{J}(dx,dy)\notag\\
+\int_\Delta \hat{f}(x)\hat{g}(x)\hat{\kappa}(dx)
\end{multline}
for any $\hat{f},\hat{g}\in\hat{\mathcal{B}}(\mathcal{E})$, where $\hat{\mathcal{E}}^c$ is a symmetric nonnegative definite bilinear form on $\hat{\mathcal{B}}(\mathcal{E})$ that is strongly local, $\hat{J}$ is a symmetric nonnegative Radon measure on $\Delta\times \Delta\setminus \left\lbrace (x,x): x\in\Delta\right\rbrace$, and $\hat{k}$ is a nonnegative Radon measure on $\Delta$. The normal contraction operates on $\hat{\mathcal{E}}^c$, and the triple $(\hat{\mathcal{E}}^c, \hat{J}, \hat{k})$ is uniquely determined.
\end{theorem}

For a proof see for instance \cite{Allain} or \cite{FOT94}. 

\begin{remark}
Note that these proofs require the local compactness but not the second countability of $\Delta$. However, if $\mathcal{B}(\mathcal{E})$ has a countable subset from which any element in $\mathcal{B}(\mathcal{E})$ can be produced by linear operations, multiplication, truncation by $1$ and taking uniform limits, then $\Delta$ is second countable and by Urysohn's theorem there exists a metric turning $\Delta$ into a locally compact separable metric space.
\end{remark}

Another result is the \emph{existence of energy measures} for the transferred Dirichlet form $(\hat{\mathcal{E}}, \hat{\mathcal{F}})$, which is an immediate consequence its regularity, \cite{FOT94, LJ78}.

\begin{theorem}\label{T:energymeasures}
For any $\hat{f}\in\hat{\mathcal{B}}(\mathcal{E})$ there exists a uniquely determined finite nonnegative Radon measure $\hat{\Gamma}(\hat{f})$ on $\Delta$ such that 
\[2\int_\Delta \hat{\varphi} d\hat{\Gamma}(\hat{f})=2\hat{\mathcal{E}}(\hat{\varphi}\hat{f},\hat{f})-\hat{\mathcal{E}}(\hat{\varphi},\hat{f}^2)\]
for any $\hat{\varphi}\in\hat{\mathcal{B}}(\mathcal{E})$. 
\end{theorem}

If the original Dirichlet form $(\mathcal{E},\mathcal{F})$ itself admits energy measures, that is if for any $f\in\mathcal{B}(\mathcal{E})$ there exists some nonnegative measure $\Gamma(f)$ such that 
\begin{equation}\label{E:originalenergy}
2\int_X \varphi d\Gamma(f)=2\mathcal{E}(\varphi f,f)-\mathcal{E}(\varphi,f^2), \ \ \varphi\in\mathcal{B}(\mathcal{E}),
\end{equation}
then the energy measures $\hat{\Gamma}(\hat{f})$ are consistent with these original ones.

\begin{theorem}
Assume that $(\mathcal{E},\mathcal{F})$ admits energy measures (\ref{E:originalenergy}). Then for any $f\in\mathcal{B}(\mathcal{E})$ we have 
\[(\Gamma(f))^\wedge=\hat{\Gamma}(\hat{f}).\]
\end{theorem}

\begin{proof}
For any $\hat{\varphi}\in C_0(\Delta)$ we have 
\begin{align}
\int_\Delta\hat{\varphi} d(\Gamma(f))^\wedge &=\int_X\varphi d\Gamma(f)\notag\\
&=2\mathcal{E}(f\varphi,f)-\mathcal{E}(\varphi, f^2)\notag\\
&=2\hat{\mathcal{E}}((f\varphi)^\wedge,\hat{f})-\hat{\mathcal{E}}(\hat{\varphi},(f^2)^\wedge)\notag\\
&=2\hat{\mathcal{E}}(\hat{f}\hat{\varphi},\hat{f})-\hat{\mathcal{E}}(\hat{\varphi},\hat{f}^2)\notag\\
&=\int_\Delta \hat{\varphi} d\hat{\Gamma}(\hat{f}).\notag
\end{align}
\end{proof}

Theorem \ref{T:energymeasures} is significant, because as the following examples show, the original Dirichlet form $(\mathcal{E},\mathcal{F})$ itself may not admit energy measures.

\begin{examples}\label{Ex:counter}
Consider the classical Dirichlet integral on the unit interval $[0,1]$, given by 
\[\mathcal{E}_0(g):=\int_0^1g'(x)^2dx\]
for any function $g$ from
\[\mathcal{F}_0:=\left\lbrace g\in C([0,1]): \mathcal{E}(g)<\infty\right\rbrace.\]
The form $(\mathcal{E}_0,\mathcal{F}_0)$ is a resistance form on $[0,1]$ in the sense of Kigami \cite{Ki03, Ki12}. We consider the countable state space $X=\mathbb{Q}\cap [0,1]$. Set
\[\mathcal{F}_0|_X:=\left\lbrace f:X\to\mathbb{R}: \text{there exists some $g\in\mathcal{F}_0$ such that $f=g|_X$}\right\rbrace\]
and 
\[\mathcal{E}(f):=\mathcal{E}_0(g), \ \ f\in \mathcal{F}_0|_X.\]
Here $g|_X$ denotes the pointwise restriction of the continuous function $g$ to $X$. By continuity and the density of $X$ in $[0,1]$ each $f\in\mathcal{F}_0|_X$ is the restriction of exactly one function $g\in\mathcal{F}_0$. Now let $\delta_q$ denote the normed Dirac point measure at a given point $q$ and let $\left\lbrace q_n\right\rbrace_{n=1}^\infty$ be an enumeration of $X$. Then 
\[\mu:=\sum_{n=1}^\infty 2^{-n} \delta_{q_n}\]
is a probability measure. The form $(\mathcal{E},\mathcal{F}_0|_X)$ is closable in $L_2(X,\mu)$, see for instance \cite[Lemma 9.2 and Theorem 9.4]{Ki12}, and its closure $(\mathcal{E},\mathcal{F})$ is a Dirichlet form. For a function $f\in \mathcal{F}_0|_X$ with $\mathcal{E}(f)>0$ (such as for instance the restriction to $X$ of a nonconstant linear function) and $g\in\mathcal{F}_0$ is such that $f=g|_X$ we have 
\begin{equation}\label{E:tryenergy}
2\mathcal{E}(\varphi f, f)-\mathcal{E}(\varphi, f^2)= 2\int_0^1 \psi (x) g'(x)^2dx,
\end{equation}
for all $\varphi\in\mathcal{F}_0|_X$ with $\varphi= \psi|_X$, $\psi\in \mathcal{F}_0$. On the other hand approximation by piecewise linear functions shows that $\mathcal{F}_0$ is dense in $C([0,1])$, and consequently any bounded Borel function on $[0,1]$ can be approximated pointwise by a uniformly bounded sequence of functions from $\mathcal{F}_0$. Let $(\psi_n)_n\subset\mathcal{F}_0$ be a uniformly bounded sequence of functions that approximate $\mathbf{1}_X$ pointwise. If for some $f$ as above $(\mathcal{E},\mathcal{F})$ would admit energy measures as in (\ref{E:originalenergy}) then we would obtain
\[\int_X \psi_n|_X\:d\Gamma(f)=\int_\Delta (\psi_n|_X)^\wedge\:d\hat{\Gamma}(f)=\int_0^1\psi_n(x) g'(x)^2 dx,\]
and by bounded convergence
\[\mathcal{E}(f)=\Gamma(f)(X)=\int_X g'(x)^2dx=0,\]
because the restriction of $g'(x)^2dx$ to $X$ is the zero measure. This contradicts $\mathcal{E}(f)>0$.
\end{examples}

\begin{remark}\label{R:infdim}
In some sense the situation of Example \ref{Ex:counter} displays a similar feature as we encounter it for Dirichlet forms on infinite dimensional spaces. For instance, let $(E,H,\mu)$ be an abstract Wiener space, cf. \cite{BH91,Gross1,Gross2, MR92,Sudakov}, let 
\[\mathcal{F}C_b^\infty:=\left\lbrace f(l_1,...,l_n): n\in\mathbb{N}, f\in C_b^\infty(\mathbb{R}^n), l_1,...,l_n\in E'\right\rbrace,\]
\[\left\langle \nabla u(z),h\right\rangle_H:=\frac{\partial u}{\partial h}(z),\ \ h\in H,\]
for any $u\in \mathcal{F}C_b^\infty$ and 
\[\mathcal{E}(u):=\int_E \left\|\nabla u\right\|_H^2 d\mu.\]
Then $(\mathcal{E},\mathcal{F}C_b^\infty)$ is closable on $L_2(E,\mu)$ and its closure $(\mathcal{E},\mathcal{F})$ is a Dirichlet form. Its energy measure is given by $\left\|\nabla u\right\|_H^2 d\mu$ on $E$. However, as the Gaussian measure $\mu$ is quasi-invariant under translations by elements of the (infinite dimensional) generalized Cameron-Martin space $H$, the space $H$ has zero Gaussian measure, hence zero energy measure. In other words, the space $H$ is too small to carry a nontrivial energy measure, but on the larger space $E$ the energy measures generally are nontrivial.
\end{remark}

\section{Separation of points and separable resistance forms}\label{S:seppoints}

In addition to Assumption \ref{A:nonvanish} respectively \ref{A:nonvanishenergy} we now assume the following.

\begin{assumption}\label{A:separate}
The space $\mathcal{B}$ separates points, that is for each $x,y\in X$, there are $f\in\mathcal{B}$ such that $f(x) \neq f(y)$. 
\end{assumption}

An immediate consequence of this assumption is that $\iota:X\to \Delta$ is injective, so $X$ is embedded in $\Delta$ as $\iota(X)$. Thus we will use $X$ and $\iota(X)$ interchangeably. 
We further assume that $\iota(X)$ is a Borel set with respect to the Gelfand topology in $\Delta$, although this assumption is technical and often can be weakened or eliminated, depending on the situation.  

\begin{remark}\mbox{}
\begin{enumerate}
\item[(i)] Assumption \ref{A:separate} leads to a situation similar to the one in \cite[Section 2]{AR89}. See also the references cited there.
\item[(ii)] If $\mathcal{B}$ does not separate points, one can define an an equivalence relation $\sim$ on $X$ by $x\sim y$ if $f(x) = f(y)$ for all $f\in\mathcal{B}$. Then all functions in $\mathcal{B}$ naturally define functions on the quotient space $\tilde X = X/\sim$, and functions in $\mathcal{B}$ separates equivalent classes. In this case, $\tilde\iota: \tilde{X}\to\Delta$, defined by $\tilde\iota([x]) = \iota(x)$ is an embedding with $\tilde\iota(\tilde X) = \iota(X)$.
\end{enumerate} 
\end{remark}

In light of \ref{T:transfermeasure}, any $\mu\in\mathcal{D}(\mathcal{B})$ can be extended to a positive measure on $\Delta$. By Assumption \ref{A:separate} we may consider the measure of $X$ in $\Delta$. In particular, for any $\sigma$-finite $\mu\in\mathcal{D}(\mathcal{B})$, we can extend $\mu$ to $\Delta$ either by considering $\hat\mu$, or by
\[
\nu(A) = \mu(A\cap X).
\] 
However, by equation \eqref{E:integralscoincide} and the Riesz representation theorem, $\hat\mu$ and $\nu$ coincide. 

The fact that $X$ is a set of full measure $\hat\mu$ allows us a technique for extending results for Dirichlet forms on locally compact spaces to a more general class of spaces. The following result is a version of \cite[Theorem 2.7]{AR89}.

\begin{proposition}
Since $\hat{\mathcal{E}}$ is a regular Dirichlet form on $\Delta$, there is a $\hat\mu$-symmetric Hunt process on $\Delta$ with Dirichlet form $\hat{\mathcal{E}}$. Since $A(\mathcal{B}(\mathcal{E}))$ separates points, $X$ is naturally identified as a subset of $\Delta$ with full $\hat\mu$-measure. By \cite[Lemma 4.1.1]{FOT94}, this implies that the process on $\Delta$ is contained in $X$ with probability 1, thus can be thought of as a process on $X$. 
\end{proposition} 

Note  that we do not claim that this process is a Hunt process on $X$ because we do not consider 
$X$ as a topological space. However the random process is well defined, which is useful in some applications such as the following.

In  what follows we will consider a special class of Dirichlet forms, the resistance forms of Kigami~\cite{Ki,Ki03,Ki12},  for which points have positive capacity.  For simplicity we define these forms in the separable case, which can be essentially reduced to a form on a countable set. 

\def\E{\mathcal E}

\begin{definition}\label{def-resistform}
A pair $(\E,\mathcal F)$ is called a resistance form on a countable set $V_*$ if it satisfies:
\begin{itemize}
\item[(RF1)] $\mathcal F$ is a linear subspace of the functions $V_{\ast}\to\mathbb{R}$ that contains the constants,
 $\E$ is a nonnegative symmetric quadratic form on $\mathcal F$, and $\E(u,u)=0$ if and only if $u$ is constant.
\item[(RF2)] The quotient of $\mathcal F$ by constant functions is Hilbert space with the norm $\E(u,u)^{1/2}$.
\item[(RF3)] If $v$ is a function on a finite set $V\subset V_*$ then there is $u\in\mathcal F$ with $u\big|_V=v$.
\item[(RF4)] For any $x,y\in V_*$ the effective resistance between $x$ and $y$ is
\begin{equation*}
R(x,y)=\sup\left\{ \frac{\big(u(x)-u(y)\big)^2}{\E(u,u)}:u\in\mathcal F,\E(u,u)>0\right\}<\infty.
\end{equation*}
\item[(RF5)] (Markov Property.) If $u\in\mathcal F$ then $\bar u(x)=\max(0,\min(1,u(x)))\in\mathcal F$ and $\E(\bar u,\bar u)\leqslant\E(u,u)$.
\end{itemize}
\end{definition}

The  resistance forms on countable sets are    determined by a sequence of traces on finite subsets, as in the following two propositions. 

\begin{proposition}[\protect{\cite{Ki,Ki03,Ki12}}]\label{basicfactsaboutresistmetric} Resistance forms have the following properties.
\begin{enumerate}
\item[(i)] $R(x,y)$ is a metric on $V_*$.  Functions in
$\mathcal F$ are $R$-continuous, thus have unique $R$-continuous extension to the $R$-completion $X_R$ of $V_*$.
\item[(ii)] If $U\subset V_*$ is finite then a Dirichlet form $\E_U$ on $U$ may be defined
by
\begin{equation*}
    \E_U(f,f)=\inf\{\E(g,g):g\in\mathcal F, g\big|_U=f\}
    \end{equation*}
in which the infimum is achieved at a unique $g$. The form $\E_U$ is called the trace of $\E$ on $U$, denoted
$\E_U=\text{Trace} U(\E)$.  If $U_1\subset U_2$ then $\E_{U_1}=\text{Trace} {U_1}(\E_{U_2})$.
\end{enumerate}
\end{proposition}

\begin{proposition}[\protect{\cite{Ki,Ki03,Ki12}}]\label{Dirformdeterminedbytraces}
Suppose $V_n\subset V_*$ are finite sets such that $V_n\subset
V_{n+1}$ and $\bigcup_{n=0}^\infty V_n$ is $R$-dense in
$V_*$. Then $\E_{{V_n}}(f,f)$ is non-decreasing and $\E(f,f)=\lim_{n\to\infty}\E_{{V_n}}(f,f)$ for any $f\in\mathcal F$. Hence $\E$ is uniquely defined by the sequence of finite dimensional traces $\E_{V_n}$ on ${V_n}$.

Conversely, suppose $V_n$ is an increasing sequence of finite sets each supporting a resistance form $\E_{V_n}$,
and the sequence is compatible in that each $\E_{V_n}$ is
the trace of $\E_{V_{n+1}}$ on ${V_n}$. Then there is a resistance form
$\E$ on $V_*=\bigcup_{n=0}^\infty V_n$ such that
$\E(f,f)=\lim_{n\to\infty}\E_{{V_n}}(f,f)$ for any $f\in\mathcal F$.
\end{proposition}

The following theorem follows easily from the analysis presented above. See [Chapter 5]\cite{Ki12} for discussion why the effective resistance metric is not suitable to define topology to produce a regular Dirichlet form.

\begin{theorem}
There exists a finite measure $\mu$ on $X=V_*$ such that:
\begin{enumerate}
	\item[(i)] any point of X has positive measure and any function of finite
energy is in $L^2(X,\mu)$;
\item[(ii)] $\mathcal E$ is a Dirichlet form on $L^2(X,\mu)$;
\item[(iii)] the embedding of $X$ into the Gelfand spectrum $\Delta$ yields a regular
Dirichlet form on $L^2(\Delta,\mu)$.
\end{enumerate}
\end{theorem}

Moreover, one can see that for any other finite measure $\mu$ on $V_*$ one can obtain a regular Dirichlet form on 
$L^2(\Delta,\mu)$ by  modifying the domain. However the case of infinite measures is more delicate. 
For instance, in [Chapter 5]\cite{Ki12} one can see that choosing the counting measure on $V_*$ 
may not produce a regular Dirichlet form, even though the space $X$ is compact in the 
topology induced by the set of functions of finite energy (but is not locally compact in the topology induced by the 
effective resistance metric).

\end{document}